\documentclass[a4paper,11pt]{article}
\usepackage[left=1in, right=1in, top=1.5in, bottom=1.5in]{geometry}

\usepackage{amsmath}
\usepackage{amssymb}
\usepackage{graphicx,psfrag}
\usepackage{amsfonts}
\usepackage{amscd}
\usepackage{enumerate,placeins,wrapfig,booktabs,chngcntr}

\usepackage{indentfirst}%

\RequirePackage{algorithm}
\RequirePackage{algorithmic}
\RequirePackage{color}

\usepackage{subcaption} %
\usepackage{multirow} %

\usepackage{nccmath} %

\usepackage{etoolbox} %
\patchcmd{\SetTagPlusEndMark}{$}{}{}{}
\patchcmd{\SetTagPlusEndMark}{$}{}{}{}

\usepackage{hyperref}

\allowdisplaybreaks[4] %

\newtheorem{theorem}{Theorem}[section]
\newtheorem{lemma}[theorem]{Lemma}

\newtheorem{proposition}[theorem]{Proposition}
\newtheorem{definition}[theorem]{Definition}
\newtheorem{remark}[theorem]{Remark} 

\renewenvironment{proof}{\bigskip\noindent\textbf{Proof.}\ \ }{\qed}
\def\qed{ \hfill $\square$}


\title{ %
		First and second-order optimality conditions for a bilinear controlled wave equation on an infinite horizon
 } %
\author{Redouane EL MEZEGUELDY\thanks{Corresponding author. MASI Laboratory, École Normale Supérieure,
		Sidi Mohamed Ben Abdellah University, Fez, Morocco ({redouane.elmezegueldy@usmba.ac.ma}).}		
	\and Zakarya DARDOUR\thanks{L2MASI Laboratory, Faculty of Sciences Dhar El Mahraz, Sidi Mohamed Ben Abdellah University, Fez, Morocco ({zakarya.dardour@usmba.ac.ma}).}
}

\begin{document}

	\date{\empty}
	\maketitle

	\begin{abstract}{
		This paper investigates the optimal control of a bilinear damped wave equation over an infinite time horizon. We establish the well-posedness of the controlled system and derive uniform energy estimates. The existence of optimal controls is proven by constructing a minimizing sequence. We prove that the control-to-state mapping is twice continuously Fréchet differentiable, which enables the derivation of first-order necessary optimality conditions in the form of a variational inequality and a pointwise projection formula. Furthermore, we establish second-order necessary and sufficient conditions: the nonnegativity of the Hessian of the cost functional is shown to be a necessary condition for local optimality, while the coercivity of this Hessian constitutes a sufficient condition. These results provide a complete characterization of local optimality for bilinear hyperbolic control systems over infinite time horizons on bounded spatial domains.		
		}
		
		\vspace{1em}
		\noindent\textbf{Keywords:} Bilinear control, wave equation, optimality conditions, second-order analysis.
		
		\vspace{0.5em}
		\noindent\textbf{2020 Mathematics Subject Classification:} 49K20, 49J20, 35L05, 35Q93
	\end{abstract}

\section{Introduction}
This paper investigates the optimal control of bilinear damped wave equations over infinite time horizons. The partial differential equation under study arises naturally in structural dynamics, modeling the displacement of vibrating structures under the influence of control forces. Structural dynamics encompasses a wide range of mechanical systems, including beams, plates, shells, and other flexible structures, as comprehensively surveyed in \cite{preumont2011vibration}.

A fundamental example is the single-degree-of-freedom system (SDFS), governed by
\begin{equation*}
	m\ddot{y} + c\dot{y} + ky = f(t),
\end{equation*}
where $m$ denotes the mass, $c$ the damping coefficient, $k$ the stiffness, $y$ the displacement, and $f(t)$ the applied external force (see \cite{craig2006fundamentals} for more general models).

While SDFS and their finite-dimensional extensions provide valuable insight, real structures—particularly elastic ones—exhibit spatially distributed behavior that is sensitive to variations across the domain and along boundaries. To capture these dynamics accurately and design effective control strategies, we need an infinite-dimensional formulation where displacement depends on both time and space and boundary effects are rigorously represented.

The control of such systems is of considerable interest, particularly for vibration suppression (see, e.g., \cite{onoda1991vibration, huang2023towards}). Depending on the application, control may enter the system either as an additive external force or as a multiplicative term affecting system parameters. For instance, damping control is employed in vibration isolation problems \cite{liu2008semi}, whereas modulation of stiffness is used for vibration suppression \cite{onoda1992active}. These examples illustrate the different practical mechanisms by which bilinear control arises in structural dynamics.

From a mathematical standpoint, modeling real structures as infinite-dimensional systems leads naturally to bilinear control problems for wave equations. Such problems have been studied for particular models, including Kirchhoff plates and classical wave equations \cite{bradley1994bilinear,liang1999bilinear}. However, a systematic analysis of distributed, space-time dependent bilinear controls over infinite time horizons remains largely unexplored, which motivates the present study.

Specifically, we consider the bilinear controlled system
\begin{equation}
	\label{eq:intro_state}
	\begin{cases}
		\ddot{y} + \dot{y} = \Delta y + u\,y + f, & \text{in } Q, \\
		y = 0, & \text{on } \Sigma, \\
		y(0) = y_0, \quad \dot{y}(0) = y_1, & \text{in } \Omega,
	\end{cases}
\end{equation}
where the state $y = y(t,x)$ satisfies a bilinear damped wave equation. Here, $u$ denotes the control, $f$ represents an external forcing term, and $y_0$ and $y_1$ are the initial displacement and velocity, respectively. The bilinear term $u\,y$ represents a multiplicative control action. The spatial domain $\Omega \subset \mathbb{R}^n$ is bounded with smooth boundary $\partial\Omega$, and the space-time domains are defined by $Q := (0,\infty)\times\Omega$ and $\Sigma := (0,\infty)\times\partial\Omega$. The homogeneous Dirichlet boundary conditions model structures with fixed (clamped) boundaries; extensions to Neumann or Robin conditions are possible but not pursued here.\\
Controls are subject to pointwise box constraints, defining the admissible set
\begin{equation}
	\label{eq:intro_Uad}
	\mathcal{U}_{\mathrm{ ad}} := \bigl\{ u \in L^{\infty}(Q) \,\big|\, \alpha \le u \le \beta \text{ a.e.\ in } Q \bigr\},
\end{equation}
where $\alpha$ and $\beta$ are given functions representing the lower and upper bounds of the control, respectively.\\
Our objective is to minimize the quadratic cost functional
\begin{equation}
	\label{eq:intro_cost}
	J(u) := \frac{1}{2} \int_0^{\infty} \!\! \int_{\Omega} \bigl( |y_u - y_d|^2 + \gamma |u|^2 \bigr) \, \mathrm{d}x \, \mathrm{d}t,
\end{equation}
where $y_u$ is the state corresponding to control $u$, $y_d$ is the desired state trajectory, and $\gamma > 0$ is a regularization parameter penalizing control effort. For stabilization problems, the typical choice is $y_d \equiv 0$, while for tracking applications, $y_d$ represents a reference trajectory that the structure should follow. The infinite-horizon formulation is particularly relevant for:
\begin{itemize}
	\item asymptotic stabilization of flexible structures operating over long time scales,
	\item avoiding artificial terminal boundary effects that arise in finite-horizon approximations of long-duration missions.
\end{itemize}
This leads to the infinite-horizon bilinear optimal control problem
\begin{equation}
	\label{prob:P}
	\tag{P}
	\min_{u \in \mathcal{U}_{\mathrm{ ad}}} J(u).
\end{equation}

Throughout the paper, we assume:
\begin{itemize}
	\item The initial conditions satisfy $y_0 \in H_0^1(\Omega)$ and $y_1 \in L^2(\Omega)$,
	\item The external forcing term $f$ belongs to $L^{\infty}(0,\infty;L^2(\Omega)) \cap L^2(Q)$,
	\item The desired state $y_d$ satisfies $y_d \in L^{\infty}(0,\infty;L^2(\Omega)) \cap L^2(Q)$,
	\item The control bounds satisfy $\alpha,\beta \in L^2(0,\infty;L^{\infty}(\Omega)) \cap L^{\infty}(Q)$ with $\alpha \le \beta$ almost everywhere.
\end{itemize}

Under these assumptions, our objectives are to establish well-posedness and uniform energy estimates for \eqref{eq:intro_state}, prove existence of optimal controls solving \eqref{prob:P}, analyze differentiability of the control-to-state mapping, derive first-order necessary optimality conditions, and obtain second-order necessary and sufficient optimality conditions that provide refined characterizations of local optimality.

The analysis of problem \eqref{prob:P} presents significant technical challenges due to the bilinear structure and the infinite-horizon setting. The choice of the control space $\mathcal{U}_{\mathrm{ad}} \subset L^{\infty}(Q)$ is dictated by these difficulties. In the spatial variable, if $u \in L^2(\Omega)$ only, the bilinear term $u\,y$ would not be well-defined as an element of the state space $L^2(\Omega)$ in general, which necessitates $u \in L^{\infty}(\Omega)$. For the temporal variable, when finite horizons are considered, one can typically assume $u \in L^2(0,T;L^{\infty}(\Omega))$ to investigate the problem. However, when infinite time horizons are considered, neither $L^2(0,\infty;L^{\infty}(\Omega))$ nor $L^{\infty}(Q)$ regularity alone is sufficient to deal with the problem, as deriving the estimates for $\ddot{y}$ necessitates their intersection $L^2(0,\infty;L^{\infty}(\Omega)) \cap L^{\infty}(Q)$.

The analysis of problem \eqref{prob:P} presents significant technical challenges arising from the interplay between the bilinear structure, hyperbolic dynamics, and infinite-horizon setting. The choice of control space $\mathcal{U}_{\mathrm{ad}} \subset L^{\infty}(Q)$ is dictated by these difficulties. In the spatial variable, since the state $y$ belongs to $L^2(\Omega)$, the bilinear term $u\,y$ must also lie in $L^2(\Omega)$ for well-posedness of the weak formulation. If we only assumed $u \in L^2(\Omega)$, the product $u\,y$ would generally belong to $L^1(\Omega)$ rather than $L^2(\Omega)$, which necessitates $u \in L^{\infty}(\Omega)$ at almost every time. For the temporal variable, when finite horizons are considered, one can typically assume $u \in L^2(0,T;L^{\infty}(\Omega))$ to control the bilinear term through Hölder's inequality. However, the infinite-horizon setting introduces additional complications: establishing the energy estimate \eqref{eq:energy_est_infinite_2} for $\ddot{y}$ requires controlling integrals of the form $\int_0^{\infty} \|u(t)\|_{L^{\infty}(\Omega)}^2 \|\nabla y(t)\|^2 \, \mathrm{d}t$ in the Grönwall argument, which demands $u \in L^2(0,\infty;L^{\infty}(\Omega))$. Simultaneously, uniform-in-time bounds on $\|\ddot{y}(t)\|_{H^{-1}(\Omega)}$ obtained from the identity $\ddot{y}(t) = -\Delta y(t) - \dot{y}(t) + u(t)y(t) + f(t)$ require $u \in L^{\infty}(Q)$. Consequently, neither $L^2(0,\infty;L^{\infty}(\Omega))$ nor $L^{\infty}(Q)$ alone suffices, as the $L^2$-temporal integrability is essential for the Grönwall estimate leading to \eqref{eq:energy_est_infinite_1}, while $L^{\infty}$-uniform boundedness is critical for the pointwise estimate \eqref{eq:energy_est_infinite_2}. This forces us to work in the intersection space $\mathcal{U} := L^2(0,\infty;L^{\infty}(\Omega)) \cap L^{\infty}(Q)$.

Our methodological approach proceeds by first establishing well-posedness and energy estimates on finite intervals $[0,T]$ via Galerkin approximation and Grönwall's inequality (Lemma~\ref{lem:energy_estimates_finite}), where constants depend explicitly on $T$. The key innovation lies in proving that these constants remain bounded uniformly in $T$ when $u \in \mathcal{U}$ (Theorem~\ref{lem:energy_estimates_infinite}), exploiting the monotonicity of the norms in $T$ to pass to the limit as $T \to \infty$. The adjoint equation \eqref{eq:adjoint_phi_system}, formulated as a backward problem with terminal conditions at infinity, is treated via the change of variables $s = T - t$ to obtain a forward problem, allowing application of uniform estimates and careful limit passage to ensure the asymptotic decay $\|\phi(t)\| + \|\dot{\phi}(t)\|_{H^{-1}} \to 0$ as $t \to \infty$ (Lemma~\ref{lem:adjoint_wellposed}). This decay is essential for integration by parts in deriving optimality conditions. Twice continuous Fréchet differentiability of the control-to-state mapping (Theorem~\ref{thm:differentiability}) is established via the implicit function theorem, formulating the state equation as $F(u,y) = 0$ and verifying that $\partial_y F(u,y_u)$ is bijective with bounded inverse independently of the time horizon.

These technical requirements distinguish our infinite-horizon bilinear wave problem from related work. For finite-horizon bilinear control \cite{bradley1994bilinear, liang1999bilinear, zerrik2019regional}, the space $L^2(0,T;L^{\infty}(\Omega))$ suffices and constants may depend on $T$, whereas the infinite horizon necessitates the more restrictive intersection space. For additive control of wave equations \cite{kunisch2016optimal, kroner2011semismooth}, the control enters linearly as $u$ rather than $u\,y$, avoiding the multiplicative coupling that requires $L^{\infty}$-spatial regularity, allowing use of $L^2(Q_T)$ or measure spaces. For infinite-horizon parabolic problems with bilinear control \cite{Reference14}, the smoothing effect of the heat kernel provides additional regularity, whereas the hyperbolic nature of the wave equation offers no such regularization, requiring more delicate treatment of $\ddot{y}$.

Optimal control problems for wave equations have been extensively studied, with particular attention to different control mechanisms and time horizons. Most existing work focuses on finite-horizon problems with additive controls. For instance, Kunisch et al.\ \cite{kunisch2016optimal} analyzed measure-valued optimal control problems for the linear wave equation with sparsity-inducing total variation penalties, establishing regularity results and first-order optimality conditions in all three space dimensions. Similarly, Kröner et al.\ \cite{kroner2011semismooth} investigated distributed, Neumann, and Dirichlet boundary controls with inequality constraints, employing semismooth Newton methods to achieve superlinear convergence, while Gugat and Grimm \cite{gugat2011optimal} studied Dirichlet boundary control with pointwise constraints for string stabilization problems. Discretization aspects and efficient solution methods have also received considerable attention, as exemplified by Liu and Pearson \cite{liu2020parameter}, who developed matching-type Schur complement preconditioners for the optimality system arising from space-time discretizations of wave control problems.

For bilinear control problems, where the control acts multiplicatively on the state, the literature is more limited. Early work by Bradley and Lenhart \cite{bradley1994bilinear} and Liang \cite{liang1999bilinear} established existence of optimal controls and derived first-order optimality conditions through formal differentiation of the cost functional for Kirchhoff plates and wave equations, respectively. More recently, Zerrik and El Kabouss \cite{zerrik2019regional} and Ait Aadi et al.\ \cite{ait2024regional} extended this framework to regional control objectives, where the desired state is prescribed only on a subregion of the spatial domain, while Clason et al.\ \cite{clason2021optimal} considered optimal control problems where the control enters as a coefficient in the principal part of the wave equation, employing total variation and multi-bang penalties to promote desired structural properties. In contrast to distributed controls, Bethke and Kröner \cite{bethke2018sufficient} analyzed the case of a scalar time-dependent control acting as a damping parameter, deriving second-order sufficient optimality conditions for both regular and singular cases using the Goh transformation. However, all these works are restricted to finite time horizons, and apart from \cite{bethke2018sufficient}, provide only first-order optimality conditions.

Second-order optimality analysis, which provides refined characterizations of local optimality and is essential for numerical solution methods, has been developed primarily for parabolic and elliptic systems \cite{casas2012second}. For bilinear systems, Aronna and Tröltzsch \cite{aronna2021first} established first- and second-order necessary and sufficient conditions for optimal control of the Fokker--Planck equation with bilinear control on finite horizons, while El Mezegueldy et al.\ \cite{Reference14} extended this framework to infinite-horizon problems with spatially localized controls, proving well-posedness and deriving complete second-order optimality conditions. For elliptic problems with bilinear controls, Casas et al.\ \cite{casas2024error} analyzed finite element discretizations and obtained error estimates under no-gap second-order sufficient conditions, revealing superconvergence phenomena in numerical experiments. However, second-order optimality analysis for hyperbolic systems, particularly wave equations, remains largely unexplored in the infinite-horizon context.

The present work addresses this gap by establishing both first- and second-order optimality conditions for distributed bilinear control of the damped wave equation on an infinite time horizon. To the best of our knowledge, this combination of bilinear structure, hyperbolic dynamics, and infinite-horizon analysis with complete second-order characterization has not been previously addressed. Our results provide a rigorous mathematical foundation for designing optimal adaptive control strategies in large flexible structures operating over extended time periods, with applications ranging from long-duration space missions to civil infrastructure monitoring and control.

The paper is organized as follows. Section~\ref{sec:wellposed} establishes well-posedness and regularity results for the state equation, including energy estimates on both finite and infinite horizons, and proves Lipschitz continuity of the control-to-state mapping. Section~\ref{sec:differentiability} introduces the adjoint equation, applies the implicit function theorem to show that the control-to-state mapping is twice continuously Fr\'{e}chet differentiable, and derives differentiability properties of the cost functional. Section~\ref{sec:firstorder} proves existence of optimal controls and derives first-order necessary optimality conditions. Section~\ref{sec:secondorder} develops second-order necessary and sufficient optimality conditions. Finally, Section~\ref{sec:conclusion} concludes with a summary and directions for future research.

\section{Well-posedness and regularity of the state equation}
\label{sec:wellposed}

This section establishes the mathematical foundation for the optimal control problem. We introduce a weak formulation of the state equation and provide standard energy estimates, which are essential for the subsequent optimality analysis.

We use standard notation from functional analysis and PDE theory. The inner product and norm in $L^2(\Omega)$ are denoted by $\langle \cdot, \cdot \rangle$ and $\|\cdot\|$, respectively. We work in Sobolev spaces $H^1_0(\Omega)$ and $H^{-1}(\Omega)$, where $\langle \cdot, \cdot \rangle_*$ denotes the duality pairing between $H^{-1}(\Omega)$ and $H^1_0(\Omega)$. For time-dependent functions, we use Bochner spaces $L^p(0,T;X)$ and $W^{k,p}(0,T;X)$. The space-time cylinder over a finite horizon is denoted $Q_T := (0,T) \times \Omega$. For notational simplicity, we often omit explicit dependence on the variables $(t,x)$ when no confusion arises, and we write $y_u$ for the state associated with control $u$.

\begin{definition}[Weak solution]
	\label{def:weak_solution}
	Let $u \in \mathcal{U}_{ad}$ and $T > 0$. A function 
	\[
	y \in C([0,T];H_0^1(\Omega)) \cap C^1([0,T];L^2(\Omega))
	\]
	is called a \emph{weak solution} of the state equation \eqref{eq:intro_state} on $[0,T]$ if $y(0)=y_0$, $\dot{y}(0)=y_1$, and for every $v \in H_0^1(\Omega)$ and almost every $t \in (0,T)$,
	\[
	\langle \ddot{y}(t)+\dot y(t), v \rangle_* + \langle \nabla y(t), \nabla v \rangle 
	= \langle u(t) y(t) + f(t), v \rangle.
	\]
\end{definition}

\begin{remark}
	\label{rem:initial_values}
	The initial values $y(0)$ and $\dot{y}(0)$ are well defined since
	\[
	L^2(0,T;H_0^1(\Omega)) \cap W^{2,1}(0,T;H^{-1}(\Omega)) \hookrightarrow 
	C([0,T];L^2(\Omega)) \cap C^1([0,T];H^{-1}(\Omega)).
	\]
\end{remark}

We now recall the standard Poincaré inequality.
\begin{lemma}
	\label{lem:poincare}
	Let $\Omega \subset \mathbb{R}^n$ be a bounded domain. Then there exists a constant $c_{\Omega}>0$ depending only on $\Omega$ such that
	\[
	\|v\|^2 \le c_{\Omega} \|\nabla v\|^2 \quad \text{for all } v \in H_0^1(\Omega).
	\]
\end{lemma}

The next result establishes existence and uniqueness of solutions to the wave equation.
\begin{lemma}
	\label{lem:energy_estimates_finite}
	Let $T > 0$ and assume that $y_0 \in H_0^1(\Omega)$, $y_1 \in L^2(\Omega)$, 
	$f \in L^2(Q_T) \cap L^{\infty}(0,T;L^2(\Omega))$, 
	and $u \in L^2(0,T;L^{\infty}(\Omega)) \cap L^{\infty}(Q_T)$. 
	Then there exists a unique weak solution $y$ in the sense of Definition~\ref{def:weak_solution} satisfying
	\begin{equation}
		\label{eq:energy_est_finite_1}
		\|y\|_{C([0,T];H_0^1(\Omega))}^2 
		+ \|\dot{y}\|_{C([0,T];L^2(\Omega))}^2
		\le 
		c_{u,T}\Big(
		\|y_0\|_{H_0^1(\Omega)}^2 + \|y_1\|^2 
		+ \|f\|_{L^2(Q_T)}^2
		\Big),
	\end{equation}
	and
	\begin{equation}
		\label{eq:energy_est_finite_2}
		\|\ddot{y}\|_{L^{\infty}(0,T;H^{-1}(\Omega))}^2
		\le 
		c_{u,T}'\Big(
		\|y_0\|_{H_0^1(\Omega)}^2 + \|y_1\|^2 + \|f\|_{L^2(Q_T)}^2
		\Big)
		+ 3\|f\|_{L^{\infty}(0,T;L^2(\Omega))}^2,
	\end{equation}
	where the constants are given by
	\[
	c_{u,T} := \exp\!\big(c_{\Omega}\|u\|_{L^2(0,T;L^\infty(\Omega))}^2\big),
	\]
	\[
	c_{u,T}' := 3\bigl(1 + \sqrt{c_{\Omega}}\|u\|_{L^\infty(Q_T)}\bigr)^2 
	\exp\!\big(c_{\Omega}\|u\|_{L^2(0,T;L^\infty(\Omega))}^2\big).
	\]
\end{lemma}

\begin{proof}
	We first derive the a priori estimates assuming a solution exists, then conclude with the existence and uniqueness result. Testing the weak formulation with $v = \dot{y}(t)$ yields
	\[
	\langle \ddot{y}(t), \dot{y}(t) \rangle_* + \langle \dot{y}(t), \dot{y}(t) \rangle + \langle \nabla y(t), \nabla \dot{y}(t) \rangle 
	= \langle u(t) y(t) + f(t), \dot{y}(t) \rangle.
	\]
	Since $\langle \ddot{y}, \dot{y} \rangle_* = \frac{1}{2}\frac{d}{dt}\|\dot{y}\|^2$ and $\langle \nabla y, \nabla \dot{y} \rangle = \frac{1}{2}\frac{d}{dt}\|\nabla y\|^2$, we obtain
	\[
	\frac{d}{dt}\Big(\frac{1}{2}\|\dot{y}(t)\|^2 + \frac{1}{2}\|\nabla y(t)\|^2\Big) + \|\dot{y}(t)\|^2 
	= \langle u(t) y(t), \dot{y}(t) \rangle + \langle f(t), \dot{y}(t) \rangle.
	\]
	Defining the energy functional $E(t) := \frac{1}{2}\|\dot{y}(t)\|^2 + \frac{1}{2}\|\nabla y(t)\|^2$, we have
	\[
	\frac{d}{dt}E(t) + \|\dot{y}(t)\|^2 = \langle u(t) y(t), \dot{y}(t) \rangle + \langle f(t), \dot{y}(t) \rangle.
	\]
	By the Cauchy--Schwarz and Young inequalities,
	\[
	\langle u y, \dot{y} \rangle \le \|u\|_{L^\infty}\|y\|\|\dot{y}\| 
	\le \frac{1}{2}\|u\|_{L^\infty}^2 \|y\|^2 + \frac{1}{2}\|\dot{y}\|^2,
	\]
	\[
	\langle f, \dot{y} \rangle \le \|f\|\|\dot{y}\| \le \frac{1}{2}\|f\|^2 + \frac{1}{2}\|\dot{y}\|^2.
	\]
	Hence,
	\[
	\frac{d}{dt}E(t) \le \frac{1}{2}\|u(t)\|_{L^\infty}^2 \|y(t)\|^2 + \frac{1}{2}\|f(t)\|^2.
	\]
	By Poincaré's inequality, $\|y(t)\|^2 \le c_{\Omega} \|\nabla y(t)\|^2 \le 2c_{\Omega} E(t)$, where $c_{\Omega}$ depends only on $\Omega$. Thus,
	\[
	\frac{d}{dt}2E(t) \le c_{\Omega}\|u(t)\|_{L^\infty}^2 2E(t) + \|f(t)\|^2.
	\]
	Integrating from $0$ to $t$ gives
	\[
	2E(t) \le 2E(0) + \int_0^t \|f(s)\|^2 ds +  \int_0^t c_{\Omega}\|u(s)\|_{L^\infty}^2 2E(s) ds.
	\]
	Applying the standard Grönwall inequality yields
	\[
	2E(t) \le \exp\Big(c_{\Omega}\int_0^t \|u(s)\|_{L^\infty}^2 ds\Big) \Big(2E(0) + \int_0^t \|f(s)\|^2 ds\Big).
	\]
	Since $\int_0^t \|u(s)\|_{L^\infty}^2 ds \le \|u\|_{L^2(0,T;L^\infty(\Omega))}^2$ and $\int_0^t \|f(s)\|^2 ds \le \|f\|_{L^2(0,T;L^2(\Omega))}^2$, we obtain
	\[
	2E(t) \le \exp\Big(c_{\Omega}\|u\|_{L^2(0,T;L^\infty)}^2\Big) \Big(2E(0) +  \|f\|_{L^2(0,T;L^2)}^2\Big).
	\]
	Recalling that $E(0) = \frac{1}{2}\|y_1\|^2 + \frac{1}{2}\|\nabla y_0\|^2$ and $E(t) = \frac{1}{2}\|\dot{y}(t)\|^2 + \frac{1}{2}\|\nabla y(t)\|^2$, we deduce
	\[
	\|\dot{y}(t)\|^2 + \|\nabla y(t)\|^2 \le  \exp\Big(c_{\Omega}\|u\|_{L^2(0,T;L^\infty)}^2\Big) \Big(\|y_0\|_{H_0^1}^2 + \|y_1\|^2 + \|f\|_{L^2(0,T;L^2)}^2\Big),
	\]
	Taking the supremum over $t \in [0,T]$ and square roots gives \eqref{eq:energy_est_finite_1}.

	For the estimate of $\ddot{y}$, we use the weak formulation to write
	\[
	\ddot{y}(t) = -\Delta y(t) - \dot{y}(t) + u(t)y(t) + f(t) \quad \text{in } H^{-1}(\Omega).
	\]
	Taking norms, we have
	\[
	\|\ddot{y}(t)\|_{H^{-1}} \le \|\nabla y(t)\| + \|\dot{y}(t)\| + \|u(t)\|_{L^\infty}\|y(t)\| + \|f(t)\|.
	\]
	By the Poincaré inequality (Lemma~\ref{lem:poincare}), $\|y(t)\| \le \sqrt{c_{\Omega}}\|\nabla y(t)\|$. Therefore,
	\[
	\|\ddot{y}(t)\|_{H^{-1}} \le (1 + \sqrt{c_{\Omega}}\|u(t)\|_{L^\infty})\|\nabla y(t)\| + \|\dot{y}(t)\| + \|f(t)\|.
	\]
	Squaring both sides and using $(a+b+c)^2 \le 3(a^2+b^2+c^2)$, we obtain
	\[
	\|\ddot{y}(t)\|_{H^{-1}}^2 \le 3(1 + \sqrt{c_{\Omega}}\|u(t)\|_{L^\infty})^2\|\nabla y(t)\|^2 + 3\|\dot{y}(t)\|^2 + 3\|f(t)\|^2.
	\]
	Taking the supremum over $t \in [0,T]$ and using \eqref{eq:energy_est_finite_1}, we get
	\[
	\begin{aligned}
		\|\ddot{y}\|_{L^{\infty}(0,T;H^{-1})}^2 
		&\le 3(1 + \sqrt{c_{\Omega}}\|u\|_{L^\infty(Q_T)})^2 \sup_{t\in[0,T]}\|\nabla y(t)\|^2 \\
		&\quad + 3\sup_{t\in[0,T]}\|\dot{y}(t)\|^2 + 3\|f\|_{L^{\infty}(0,T;L^2)}^2 \\
		&\le 3(1 + \sqrt{c_{\Omega}}\|u\|_{L^\infty(Q_T)})^2 \exp\Big(c_{\Omega}\|u\|_{L^2(0,T;L^\infty)}^2\Big) \\
		&\quad \times \Big(\|y_0\|_{H_0^1}^2 + \|y_1\|^2 + \|f\|_{L^2(0,T;L^2)}^2\Big) + 3\|f\|_{L^{\infty}(0,T;L^2)}^2.
	\end{aligned}
	\]
	This yields \eqref{eq:energy_est_finite_2}.
	
	The existence and uniqueness of the weak solution follow from the standard Galerkin method combined with the a priori estimates derived above.
\end{proof}

Building on the previous lemma, we obtain the following result that ensures well-posedness of the state equation over an infinite time horizon.
\begin{theorem}
	\label{lem:energy_estimates_infinite}
	Assume that $y_0 \in H_0^1(\Omega)$, $y_1 \in L^2(\Omega)$, 
	$f \in L^2(Q) \cap L^{\infty}(0,\infty;L^2(\Omega))$, 
	and $u \in L^2(0,\infty;L^{\infty}(\Omega)) \cap L^{\infty}(Q)$. 
	Then the weak solution $y$ of the state equation \eqref{eq:intro_state} satisfies
	\begin{equation}
		\label{eq:energy_est_infinite_1}
		\begin{aligned}
			\|y\|_{L^{\infty}(0,\infty;H_0^1(\Omega))}^2
			+ \|\dot{y}\|_{L^{\infty}(0,\infty;L^2(\Omega))}^2 
			\le c_u\Big(
			\|y_0\|_{H_0^1(\Omega)}^2
			+ \|y_1\|^2
			+ \|f\|_{L^2(Q)}^2
			\Big),
		\end{aligned}
	\end{equation}
	and
	\begin{equation}
		\label{eq:energy_est_infinite_2}
		\begin{aligned}
			\|\ddot{y}\|_{L^{\infty}(0,\infty;H^{-1}(\Omega))}^2
			&\le c_u'\Big(
			\|y_0\|_{H_0^1(\Omega)}^2
			+ \|y_1\|^2
			+ \|f\|_{L^2(Q)}^2
			\Big)
			+ 3\|f\|_{L^{\infty}(0,\infty;L^2(\Omega))}^2,
		\end{aligned}
	\end{equation}
	where
	\[
	c_u := \exp\!\big(c_\Omega \|u\|_{L^2(0,\infty;L^\infty(\Omega))}^2\big),
	\]
	\[
	c_u' := 3(1 + \sqrt{c_{\Omega}}\|u\|_{L^\infty(Q)})^2 \exp\!\big(c_{\Omega}\|u\|_{L^2(0,\infty;L^\infty(\Omega))}^2\big).
	\]
\end{theorem}

\begin{proof}
	From Lemma~\ref{lem:energy_estimates_finite}, for each $T>0$ we have
	\[
	\begin{aligned}
		&\|y\|_{C([0,T];H_0^1(\Omega))}^2 + \|\dot{y}\|_{C([0,T];L^2(\Omega))}^2 \\
		&\qquad \le 
		\exp\!\Big(c_\Omega\|u\|_{L^2(0,T;L^\infty(\Omega))}^2\Big)
		\Big(
		\|y_0\|_{H_0^1}^2 + \|y_1\|^2 + \|f\|_{L^2(Q_T)}^2
		\Big).
	\end{aligned}
	\]
	Since $\|u\|_{L^2(0,T;L^\infty(\Omega))}$ and $\|f\|_{L^2(Q_T)}$ are monotone in $T$, 
	we can pass to the limit $T\to\infty$. The right-hand side remains finite because 
	$u \in L^2(0,\infty;L^\infty(\Omega))$ and $f \in L^2(Q)$, yielding \eqref{eq:energy_est_infinite_1}. 
	Similarly, applying the same argument to estimate~\eqref{eq:energy_est_finite_2} yields~\eqref{eq:energy_est_infinite_2}.
\end{proof}

We close this section with a result establishing Lipschitz continuity of the solutions with respect to both controls and external forcing.
\begin{lemma}
	\label{lem:lipschitz}
	Assume the standing hypotheses hold. For $i=1,2$, let 
	\[
	u_i \in L^2(0,\infty;L^{\infty}(\Omega)) \cap L^{\infty}(Q), \quad
	f_i \in L^2(Q) \cap L^{\infty}(0,\infty;L^2(\Omega)),
	\] 
	and denote by $y_i$ the corresponding weak solutions with the same initial data $(y_0, y_1)$. Then there exists a constant $L>0$, depending on $\Omega$ and the norms of $y_i$ and $u_i$, such that
	\begin{equation}
		\label{eq:lipschitz}
		\begin{aligned}
			&\|y_1 - y_2\|_{L^{\infty}(0,\infty;H_0^1(\Omega))}^2
			+ \|\dot y_1 - \dot y_2\|_{L^{\infty}(0,\infty;L^2(\Omega))}^2
			+ \|\ddot y_1 - \ddot y_2\|_{L^{\infty}(0,\infty;H^{-1}(\Omega))}^2 \\
			&\le L\Big(
			\|u_1 - u_2\|_{L^{\infty}(Q) \cap L^2(0,\infty;L^{\infty}(\Omega))}^2
			+ \|f_1 - f_2\|_{L^{\infty}(0,\infty;L^2(\Omega)) \cap L^2(Q)}^2
			\Big).
		\end{aligned}
	\end{equation}
\end{lemma}

\begin{proof}
	Let $\delta y = y_1 - y_2$, $\delta u = u_1 - u_2$, and $\delta f = f_1 - f_2$.  
	Then $\delta y$ satisfies the linear system
	\[
	\delta \ddot y + \delta \dot y - \Delta \delta y = u_1 \delta y + \delta u \, y_2 + \delta f, 
	\quad \delta y(0)=\delta \dot y(0)=0.
	\]
	
	The desired Lipschitz estimate~\eqref{eq:lipschitz} follows directly from the energy estimate for this linear system (Lemma~\ref{lem:energy_estimates_infinite}) together with standard norm inequalities.
\end{proof}

\begin{remark}
	If $u_1, u_2 \in \mathcal{U}_{ad}$, then the constant $L$ in \eqref{eq:lipschitz} can be chosen independently of $u_1, u_2, y_1$, and $y_2$. 
	In this case, $L$ depends only on the domain $\Omega$, the initial data $(y_0,y_1)$, and the control bounds $\alpha, \beta$.
\end{remark}

\section{Differentiability analysis}
\label{sec:differentiability}
This part establishes the ingredients required for the derivation of optimality conditions. 
We introduce the adjoint equation associated with the bilinear wave dynamics, analyze the Fréchet differentiability of the control-to-state mapping, and conclude with the differentiability properties of the cost functional.

\subsection{The adjoint equation}

In what follows, we denote the control space by
\[
\mathcal{U} := L^2(0,\infty;L^\infty(\Omega)) \cap L^\infty(Q),
\]
and the state space by
\[
\mathcal{Y} := L^\infty(0,\infty;H_0^1(\Omega))
\cap W^{1,\infty}(0,\infty;L^2(\Omega))
\cap W^{2,\infty}(0,\infty;H^{-1}(\Omega)).
\]

The next lemma establishes the well-posedness of the adjoint equation, which will be instrumental in characterizing the first and second derivatives of the cost functional.
\begin{lemma}
	\label{lem:adjoint_wellposed}
	For any control $u \in \mathcal{U}$, let $y_u$ denote the corresponding state solution of \eqref{eq:intro_state}. Then there exists a unique adjoint state $\phi \in \mathcal{Y}$ satisfying
	\begin{equation}
		\label{eq:adjoint_phi_system}
		\begin{cases}
			\ddot{\phi} - \dot{\phi} = \Delta \phi + u \phi + y_u - y_d, & \text{in } Q, \\
			\phi = 0, & \text{on } \partial \Omega, \\
			\|\phi(t)\| + \|\dot{\phi}(t)\|_{H^{-1}(\Omega)} \to 0, & \text{as } t \to \infty.
		\end{cases}
	\end{equation}
	Furthermore, $\phi$ satisfies the energy estimates
	\begin{align*}
		\|\phi\|_{L^\infty(0,\infty;H_0^1(\Omega))}^2 + \|\dot{\phi}\|_{L^\infty(0,\infty;L^2(\Omega))}^2 &\le c_u \|y_u - y_d\|_{L^2(Q)}^2, \\
		\|\ddot{\phi}\|_{L^\infty(0,\infty;H^{-1}(\Omega))}^2 &\le
		c_u' \|y_u - y_d\|_{L^2(Q)}^2 + 3\|y_u - y_d\|_{L^\infty(0,\infty;L^2(\Omega))}^2,
	\end{align*}
	where $c_u$ and $c_u'$ are the constants from Lemma \ref{lem:energy_estimates_infinite}.
\end{lemma}

\begin{proof}
	Existence and uniqueness follow by truncating the adjoint equation to a finite horizon, performing the change of variable $s = T - t$ to transform it into a forward problem, applying uniform energy estimates from Lemma \ref{lem:energy_estimates_infinite}, and passing to the limit as $T \to \infty$. A similar strategy was used in \cite[Lemma~3]{Reference14}.
\end{proof}

\subsection{Differentiability of the control-to-state mapping}

Consider the control-to-state operator 
\[
S:\mathcal{U}\to\mathcal{Y},\qquad S(u)=y_u ,
\]
where $y_u$ denotes the unique solution of \eqref{eq:intro_state}. The following result establishes the differentiability properties of this operator.
\begin{theorem}	\label{thm:differentiability}
	Assume the standing hypotheses hold and let $f\in L^2(Q)\cap L^\infty(0,\infty;L^2(\Omega))$ be fixed. Then  
	\[
	S:\mathcal{U}\to\mathcal{Y}
	\]
	is twice continuously Fréchet differentiable. The derivatives are characterized as follows:
	
	\begin{enumerate}
		\item For $u,h\in\mathcal{U}$, the first derivative $z=S'(u)[h]\in\mathcal{Y}$ is the solution of
		\begin{equation}\label{eq:linearized}
			\begin{cases}
				\ddot{z}+\dot{z}=\Delta z+u z+h\,y_u, &\text{in }Q,\\ 
				z=0, &\text{on }\Sigma,\\ 
				z(0)=0,\quad \dot{z}(0)=0, &\text{in }\Omega.
			\end{cases}
		\end{equation}
		
		\item For $u,h_1,h_2\in\mathcal{U}$, the second derivative  
		$\zeta=S''(u)[h_1,h_2]\in\mathcal{Y}$ solves
		\begin{equation}\label{eq:second_linearized}
			\begin{cases}
				\ddot{\zeta}+\dot{\zeta}=\Delta\zeta+u\zeta+h_1 z_2+h_2 z_1, &\text{in }Q,\\ 
				\zeta=0, &\text{on }\Sigma,\\ 
				\zeta(0)=0,\quad \dot{\zeta}(0)=0, &\text{in }\Omega,
			\end{cases}
		\end{equation}
		where $z_i=S'(u)[h_i]$ for $i=1,2$.
	\end{enumerate}
\end{theorem}

\begin{proof}
	Define the operator $F : \mathcal{U} \times \mathcal{Y} \to \mathcal{F}$ by
	\[
	F(u, y) := \ddot{y} + \dot{y} - \Delta y - u y - f,
	\]
	where $\mathcal{F} := L^2(Q) \cap L^{\infty}(0,\infty;L^2(\Omega))$. The state equation \eqref{eq:intro_state} is equivalent to $F(u, y) = 0$.
	
	The operator $F$ is of class $C^2$ with partial derivatives
	\[
	\partial_y F(u, y)[\psi] = \ddot{\psi} + \dot{\psi} - \Delta \psi - u \psi, \quad 
	\partial_u F(u, y)[h] = -h y,
	\]
	\[
	\partial_{yy}^2 F(u, y)[\psi_1, \psi_2] = 0, \quad 
	\partial_{uy}^2 F(u, y)[h, \psi] = -h \psi, \quad
	\partial_{uu}^2 F(u, y)[h_1, h_2] = 0.
	\]
	By Lemma~\ref{lem:energy_estimates_infinite}, the operator $\partial_y F(u, y_u) : \mathcal{Y} \to \mathcal{F}$ is bijective with bounded inverse. The implicit function theorem (see \cite{troeltzsch2010optimal, aronna2021first}) guarantees that $S$ is of class $C^2$.
	
	The first derivative is obtained by differentiating $F(u, S(u)) = 0$:
	\[
	\partial_y F(u, y_u)[S'(u)[h]] + \partial_u F(u, y_u)[h] = 0,
	\]
	which gives $z = S'(u)[h]$ solving \eqref{eq:linearized}.
	
	The second derivative follows by differentiating once more:
	\[
	\partial_y F(u, y_u)[S''(u)[h_1, h_2]] = - \partial_{uy}^2 F(u, y_u)[h_1, S'(u)[h_2]] - \partial_{uy}^2 F(u, y_u)[h_2, S'(u)[h_1]],
	\]
	which gives $\zeta = S''(u)[h_1, h_2]$ solving \eqref{eq:second_linearized}.
\end{proof}

\begin{remark}
	The proof follows the approach of \cite[Corollary 4.2]{aronna2021first}; the linearized equations \eqref{eq:linearized}--\eqref{eq:second_linearized} differ slightly, but the differentiability results remain valid, which justifies the derivatives above.
\end{remark}

\subsection{Differentiability of the cost functional}

Using the differentiability properties of the control-to-state mapping established in Theorem~\ref{thm:differentiability}, we now characterize the derivatives of the cost functional in terms of the adjoint state.

\begin{proposition}
	\label{prop:cost_differentiability}
	The cost functional $J : \mathcal{U} \to \mathbb{R}$ defined by \eqref{eq:intro_cost} is twice continuously Fréchet differentiable. Moreover, for $u \in \mathcal{U}$, let $y_u = S(u)$ and $\phi_u \in \mathcal{Y}$ denote the adjoint state solving \eqref{eq:adjoint_phi_system}. Then, for $h_1, h_2 \in \mathcal{U}$, the following hold:
	
	\begin{enumerate}
		\item The first derivative is
		\begin{equation}
			\label{eq:cost_first_derivative}
			J'(u)[h] = \int_0^\infty \int_\Omega (\phi_u y_u + \gamma u)\, h \, \mathrm{d}x \, \mathrm{d}t.
		\end{equation}
		
		\item The second derivative is
		\begin{equation}
			\label{eq:cost_second_derivative}
			J''(u)[h_1, h_2] = \int_0^\infty \int_\Omega \Big( z_1 z_2 + \phi_u (h_1 z_2 + h_2 z_1) + \gamma h_1 h_2 \Big)\, \mathrm{d}x \, \mathrm{d}t,
		\end{equation}
		where $z_i = S'(u)[h_i]$ for $i = 1,2$, and each $z_i$ solves the linearized state equation \eqref{eq:linearized}.
	\end{enumerate}
\end{proposition}

\begin{proof}
	\textbf{Twice continuous differentiability.} 
	Since $S : \mathcal{U} \to \mathcal{Y}$ is twice continuously Fréchet differentiable by Theorem~\ref{thm:differentiability}, and the mapping $y \mapsto \frac{1}{2}\|y - y_d\|_{L^2(Q)}^2$ is twice continuously Fréchet differentiable on $\mathcal{Y}$, the chain rule for Fréchet derivatives yields the twice continuous differentiability of $J$.
	
	\textbf{First derivative.} 
	By direct differentiation of the cost functional, for $h \in \mathcal{U}$, we have
	\begin{equation}
		\label{eq:cost_derivative_initial}
		J'(u)[h] = \int_0^\infty \int_\Omega (y_u - y_d) z \, \mathrm{d}x \, \mathrm{d}t + \gamma \int_0^\infty \int_\Omega u h \, \mathrm{d}x \, \mathrm{d}t,
	\end{equation}
	where $z = S'(u)[h]$ satisfies the linearized state equation \eqref{eq:linearized}:
	\[
	\ddot{z} + \dot{z} = \Delta z + u z + h y_u, \quad z(0) = \dot{z}(0) = 0.
	\]
	
	We now express the first integral in terms of the adjoint state. Write the integral as a limit:
	\[
	\int_0^\infty \int_\Omega (y_u - y_d) z \, \mathrm{d}x \, \mathrm{d}t = \lim_{T \to \infty} \int_0^T \int_\Omega (y_u - y_d) z \, \mathrm{d}x \, \mathrm{d}t.
	\]
	
	Multiply the linearized state equation by the adjoint state $\phi_u$ and integrate over $(0,T) \times \Omega$:
	\[
	\int_0^T \int_\Omega (\ddot{z} + \dot{z}) \phi_u \, \mathrm{d}x \, \mathrm{d}t = \int_0^T \int_\Omega (\Delta z + u z + h y_u) \phi_u \, \mathrm{d}x \, \mathrm{d}t.
	\]
	
	For the left-hand side, integrate by parts in time:
	\begin{align*}
		\int_0^T \int_\Omega \ddot{z} \phi_u \, \mathrm{d}x \, \mathrm{d}t 
		&= \left[\int_\Omega \dot{z} \phi_u \, \mathrm{d}x\right]_0^T - \int_0^T \int_\Omega \dot{z} \dot{\phi}_u \, \mathrm{d}x \, \mathrm{d}t \\
		&= \int_\Omega \dot{z}(T) \phi_u(T) \, \mathrm{d}x - \int_0^T \int_\Omega \dot{z} \dot{\phi}_u \, \mathrm{d}x \, \mathrm{d}t,
	\end{align*}
	where we used $z(0) = 0$. Similarly,
	\begin{align*}
		\int_0^T \int_\Omega \dot{z} \dot{\phi}_u \, \mathrm{d}x \, \mathrm{d}t 
		&= \left[\int_\Omega z \dot{\phi}_u \, \mathrm{d}x\right]_0^T - \int_0^T \int_\Omega z \ddot{\phi}_u \, \mathrm{d}x \, \mathrm{d}t \\
		&= \int_\Omega z(T) \dot{\phi}_u(T) \, \mathrm{d}x - \int_0^T \int_\Omega z \ddot{\phi}_u \, \mathrm{d}x \, \mathrm{d}t.
	\end{align*}
	
	Therefore,
	\[
	\int_0^T \int_\Omega \ddot{z} \phi_u \, \mathrm{d}x \, \mathrm{d}t = \int_\Omega \dot{z}(T) \phi_u(T) \, \mathrm{d}x - \int_\Omega z(T) \dot{\phi}_u(T) \, \mathrm{d}x + \int_0^T \int_\Omega z \ddot{\phi}_u \, \mathrm{d}x \, \mathrm{d}t.
	\]
	
	For the first-order time derivative term:
	\[
	\int_0^T \int_\Omega \dot{z} \phi_u \, \mathrm{d}x \, \mathrm{d}t = \int_\Omega z(T) \phi_u(T) \, \mathrm{d}x - \int_0^T \int_\Omega z \dot{\phi}_u \, \mathrm{d}x \, \mathrm{d}t.
	\]
	
	For the spatial term, integrate by parts (using $z = \phi_u = 0$ on $\partial\Omega$):
	\[
	\int_0^T \int_\Omega \Delta z \, \phi_u \, \mathrm{d}x \, \mathrm{d}t = \int_0^T \int_\Omega z \, \Delta \phi_u \, \mathrm{d}x \, \mathrm{d}t.
	\]
	
	Combining all terms:
	\begin{align*}
		&\int_0^T \int_\Omega z (\ddot{\phi}_u - \dot{\phi}_u - \Delta \phi_u - u \phi_u) \, \mathrm{d}x \, \mathrm{d}t \\
		&\quad = \int_0^T \int_\Omega h y_u \phi_u \, \mathrm{d}x \, \mathrm{d}t - \int_\Omega \dot{z}(T) \phi_u(T) \, \mathrm{d}x \\
		&\qquad + \int_\Omega z(T) \dot{\phi}_u(T) \, \mathrm{d}x - \int_\Omega z(T) \phi_u(T) \, \mathrm{d}x.
	\end{align*}
	
	Using the adjoint equation \eqref{eq:adjoint_phi_system}, we have $\ddot{\phi}_u - \dot{\phi}_u - \Delta \phi_u - u \phi_u = y_u - y_d$, so
	\[
	\int_0^T \int_\Omega z (y_u - y_d) \, \mathrm{d}x \, \mathrm{d}t = \int_0^T \int_\Omega h y_u \phi_u \, \mathrm{d}x \, \mathrm{d}t + \text{boundary terms at } T.
	\]
	
	By Lemma~\ref{lem:adjoint_wellposed}, $\|\phi_u(T)\| + \|\dot{\phi}_u(T)\|_{H^{-1}(\Omega)} \to 0$ as $T \to \infty$. Moreover, from Lemma~\ref{lem:energy_estimates_infinite}, $z$ remains bounded in $L^\infty(0,\infty;H_0^1(\Omega)) \cap W^{1,\infty}(0,\infty;L^2(\Omega))$. Therefore, the boundary terms vanish as $T \to \infty$.
	
	Taking the limit and using dominated convergence, we obtain:
	\[
	\int_0^\infty \int_\Omega z (y_u - y_d) \, \mathrm{d}x \, \mathrm{d}t = \int_0^\infty \int_\Omega h y_u \phi_u \, \mathrm{d}x \, \mathrm{d}t.
	\]
	
	Substituting into \eqref{eq:cost_derivative_initial} yields
	\[
	J'(u)[h] = \int_0^\infty \int_\Omega (\phi_u y_u + \gamma u) h \, \mathrm{d}x \, \mathrm{d}t,
	\]
	which is \eqref{eq:cost_first_derivative}.
	
	\textbf{Second derivative.} 
	To compute the second derivative, we differentiate the expression \eqref{eq:cost_derivative_initial} with respect to the control. For $h_1, h_2 \in \mathcal{U}$, let $z_1 = S'(u)[h_1]$ and $z_2 = S'(u)[h_2]$ satisfy \eqref{eq:linearized}. By the product rule and the twice continuous differentiability of $S$, we have
	\[
	J''(u)[h_1, h_2] = \int_0^\infty \int_\Omega z_2 z_1 \, \mathrm{d}x \, \mathrm{d}t + \int_0^\infty \int_\Omega (y_u - y_d) \zeta \, \mathrm{d}x \, \mathrm{d}t + \gamma \int_0^\infty \int_\Omega h_2 h_1 \, \mathrm{d}x \, \mathrm{d}t,
	\]
	where $\zeta = S''(u)[h_1, h_2]$ satisfies the second linearized equation \eqref{eq:second_linearized}:
	\[
	\ddot{\zeta} + \dot{\zeta} = \Delta \zeta + u \zeta + h_1 z_2 + h_2 z_1, \quad \zeta(0) = \dot{\zeta}(0) = 0.
	\]
	
	Following the same reasoning as in the derivation of the first derivative, we multiply the equation for $\zeta$ by the adjoint state $\phi_u$ and integrate by parts over $(0,T) \times \Omega$. Using the adjoint equation \eqref{eq:adjoint_phi_system} and the terminal conditions $\|\phi_u(T)\|_{H_0^1} + \|\dot{\phi}_u(T)\| \to 0$ as $T \to \infty$, we obtain
	\[
	\int_0^\infty \int_\Omega (y_u - y_d) \zeta \, \mathrm{d}x \, \mathrm{d}t = \int_0^\infty \int_\Omega \phi_u (h_1 z_2 + h_2 z_1) \, \mathrm{d}x \, \mathrm{d}t.
	\]
	
	Substituting this into the expression for $J''(u)[h_1, h_2]$ yields
	\[
	J''(u)[h_1, h_2] = \int_0^\infty \int_\Omega \Big( z_1 z_2 + \phi_u (h_1 z_2 + h_2 z_1) + \gamma h_1 h_2 \Big) \, \mathrm{d}x \, \mathrm{d}t,
	\]
	which is \eqref{eq:cost_second_derivative}.
\end{proof}

The following lemma establishes key continuity and boundedness properties of the derivatives of the cost functional, which are essential for analyzing optimality conditions.
\begin{lemma}\label{lem:cost_derivative_estimates}
	There exists a positive constant $C$ such that for all $u_1, u_2 \in \mathcal{U}_{\mathrm{ad}}$ and $h_1, h_2 \in \mathcal{U}$, the following estimates hold:
	\begin{align}
		|J'(u_1)[h_1]| &\leq C \|h_1\|_{\mathcal{U}}, \label{lem:estimate1}\\
		|J''(u_1)[h_1, h_2]| &\leq C \|h_1\|_{\mathcal{U}} \|h_2\|_{\mathcal{U}}, \label{lem:estimate2}\\
		|J'(u_1)[h_1] - J'(u_2)[h_1]| &\leq C \|u_1 - u_2\|_{\mathcal{U}} \|h_1\|_{\mathcal{U}}, \label{lem:estimate3}\\
		|J''(u_1)[h_1, h_1] - J''(u_2)[h_1, h_1]| &\leq C \|u_1 - u_2\|_{\mathcal{U}} \|h_1\|_{\mathcal{U}}^2. \label{lem:estimate4}
	\end{align}
\end{lemma}

\begin{proof}
	The estimates follow from the expressions for $J'(u)$ and $J''(u)$ established in Proposition~\ref{prop:cost_differentiability}, combined with the energy estimates from Lemmas~\ref{lem:energy_estimates_infinite} and~\ref{lem:adjoint_wellposed}, and the Lipschitz continuity from Lemma~\ref{lem:lipschitz}.
	
	\textbf{Estimate \eqref{lem:estimate1}:} From \eqref{eq:cost_first_derivative}, we have
	\[
	|J'(u_1)[h_1]| \le \int_0^\infty \int_\Omega |\phi_{u_1} y_{u_1}| |h_1| \, \mathrm{d}x \, \mathrm{d}t + \gamma \int_0^\infty \int_\Omega |u_1| |h_1| \, \mathrm{d}x \, \mathrm{d}t.
	\]
	By Hölder's inequality and the energy estimates in Lemmas~\ref{lem:energy_estimates_infinite} and~\ref{lem:adjoint_wellposed}, both $\phi_{u_1}$ and $y_{u_1}$ remain bounded in $L^\infty(0,\infty;L^2(\Omega))$ with bounds depending only on the data. Since $u_1 \in \mathcal{U}_{\mathrm{ad}}$, we obtain \eqref{lem:estimate1} with $C$ depending on $\Omega$, the initial data, $\alpha$, $\beta$, $f$, and $y_d$.
	
	\textbf{Estimate \eqref{lem:estimate2}:} From \eqref{eq:cost_second_derivative},
	\[
	|J''(u_1)[h_1, h_2]| \le \int_0^\infty \int_\Omega (|z_1 z_2| + |\phi_{u_1}|(|h_1 z_2| + |h_2 z_1|) + \gamma |h_1 h_2|) \, \mathrm{d}x \, \mathrm{d}t,
	\]
	where $z_i = S'(u_1)[h_i]$. By the energy estimate \eqref{eq:energy_est_infinite_1} applied to the linearized equation \eqref{eq:linearized}, we have $\|z_i\|_{L^\infty(0,\infty;L^2(\Omega))} \le C \|h_i\|_{\mathcal{U}}$. Combined with the boundedness of $\phi_{u_1}$ from Lemma~\ref{lem:adjoint_wellposed}, estimate \eqref{lem:estimate2} follows.
	
	\textbf{Estimate \eqref{lem:estimate3}:} Write
	\[
	J'(u_1)[h_1] - J'(u_2)[h_1] = \int_0^\infty \int_\Omega [(\phi_{u_1} y_{u_1} - \phi_{u_2} y_{u_2}) + \gamma(u_1 - u_2)] h_1 \, \mathrm{d}x \, \mathrm{d}t.
	\]
	The Lipschitz continuity of $u \mapsto y_u$ (Lemma~\ref{lem:lipschitz}) and $u \mapsto \phi_u$ (which follows from the adjoint equation by similar arguments) yields
	\[
	\|y_{u_1} - y_{u_2}\|_{L^\infty(0,\infty;L^2(\Omega))} + \|\phi_{u_1} - \phi_{u_2}\|_{L^\infty(0,\infty;L^2(\Omega))} \le C \|u_1 - u_2\|_{\mathcal{U}}.
	\]
	Applying Hölder's inequality gives \eqref{lem:estimate3}.
	
	\textbf{Estimate \eqref{lem:estimate4}:} From \eqref{eq:cost_second_derivative}, we compute
	\begin{align*}
		&J''(u_1)[h_1, h_1] - J''(u_2)[h_1, h_1] \\
		&= \int_0^\infty \int_\Omega \Big[(z_1^{(1)})^2 - (z_1^{(2)})^2 + (\phi_{u_1} - \phi_{u_2}) \cdot 2h_1 z_1^{(1)} \\
		&\qquad\qquad\qquad + \phi_{u_2} \cdot 2h_1(z_1^{(1)} - z_1^{(2)})\Big] \, \mathrm{d}x \, \mathrm{d}t,
	\end{align*}
	where $z_1^{(i)} = S'(u_i)[h_1]$ for $i=1,2$. Using $(z_1^{(1)})^2 - (z_1^{(2)})^2 = (z_1^{(1)} + z_1^{(2)})(z_1^{(1)} - z_1^{(2)})$ and applying Lemma~\ref{lem:lipschitz} to bound $\|z_1^{(1)} - z_1^{(2)}\|$ and $\|\phi_{u_1} - \phi_{u_2}\|$ in terms of $\|u_1 - u_2\|_{\mathcal{U}}$, we obtain \eqref{lem:estimate4}.
\end{proof}

\section{Existence and first-order optimality conditions}
\label{sec:firstorder}
This section establishes existence of optimal controls and derives the corresponding first-order necessary optimality conditions.

\subsection{Existence of optimal controls}

The following result guarantees that the optimal control problem \eqref{prob:P} admits at least one solution.

\begin{theorem}
	\label{thm:existence_optimal_control}
	The optimal control problem \eqref{prob:P} admits at least one solution.
\end{theorem}

\begin{proof}
	Let $(u_k)_{k\in\mathbb{N}}\subset \mathcal{U}_{\mathrm{ ad}}$ be a minimizing sequence, i.e.
	\[J(u_k) \to \inf_{u\in\mathcal{U}_{\mathrm{ ad}}} J(u)\quad\text{as }k\to\infty.\]
	Since each $u_k$ satisfies $\alpha\le u_k\le\beta$ a.e.\ in $Q$ and $\alpha,\beta\in L^2(0,\infty;L^{\infty}(\Omega)) \cap L^{\infty}(Q)$, the sequence $(u_k)$ is uniformly bounded in $L^2(0,\infty;L^{\infty}(\Omega)) \cap L^{\infty}(Q)$. By the Banach--Alaoglu theorem there exists a subsequence (still denoted by $(u_k)$) and a limit $\bar{u}\in L^2(0,\infty;L^{\infty}(\Omega)) \cap L^{\infty}(Q)$ such that
	\[u_k \stackrel{*}{\rightharpoonup} \bar{u}\quad\text{in }L^{\infty}(Q)\text{ as }k\to\infty.
	\]
	The bounds are preserved in the limit, hence $\bar{u}\in\mathcal{U}_{ad}$.
	
	Let $y_k:=y_{u_k}$ be the associated states. From Lemma~\ref{lem:energy_estimates_infinite} the sequence $(y_k)$ is bounded in
	\[L^{\infty}(0,\infty;H_0^1(\Omega))\cap W^{1,\infty}(0,\infty;L^2(\Omega)),\]
	and the sequence $(\ddot y_k)$ is bounded in $L^{\infty}(0,\infty;H^{-1}(\Omega))$. Fix $T>0$. By the Aubin--Lions compactness lemma (applied on $(0,T)$) there exists a subsequence (still denoted $(y_k)$) and a limit $\bar{y}$ such that
	\[y_k \rightharpoonup \bar{y}\quad\text{in }L^{\infty}(0,T;H_0^1(\Omega)),\qquad y_k \to \bar{y}\quad\text{in }L^2(0,T;L^2(\Omega)).\]
	Moreover, $\dot y_k \rightharpoonup \dot{\bar{y}}$ in $L^{\infty}(0,T;L^2(\Omega))$ and $\ddot y_k \rightharpoonup \ddot{\bar{y}}$ in $L^{\infty}(0,T;H^{-1}(\Omega))$.
	
	The limit pair $(\bar{u},\bar{y})$ satisfies the weak formulation of the state equation on every finite interval $(0,T)$: passing to the limit in the weak formulation for $(u_k,y_k)$ is justified since
	\begin{itemize}
		\item linear terms converge by weak (or weak-*) convergence,
		\item the term $u_k y_k$ converges to $\bar{u}\,\bar{y}$ in the sense of distributions on $(0,T)\times\Omega$ because $y_k\to\bar{y}$ strongly in $L^2(0,T;L^2(\Omega))$ and $u_k\stackrel{*}{\rightharpoonup}\bar{u}$ in $L^{\infty}(0,T;L^{\infty}(\Omega))$, hence
		\[u_k y_k \rightharpoonup \bar{u}\,\bar{y}\quad\text{in }L^2(0,T;L^2(\Omega)).\]
	\end{itemize}
	Therefore $\bar{y}=y_{\bar{u}}$ is the weak solution corresponding to $\bar{u}$.
	
	It remains to show that $\bar{u}$ is optimal. The cost functional is nonnegative, and by lower semicontinuity of the $L^2$-norm with respect to the convergences above (together with dominated convergence on finite intervals and monotone convergence on $(0,\infty)$) it holds that
	\[J(\bar{u}) \le \liminf_{k\to\infty} J(u_k) = \inf_{u\in\mathcal{U}_{\mathrm{ ad}}} J(u).\]
	Hence $\bar{u}$ is a minimizer of \eqref{prob:P}.
\end{proof}

\subsection{First-order necessary conditions}

We now derive the first-order necessary conditions for the optimal control problem \eqref{prob:P}.

\begin{proposition}[First-order necessary conditions]
	\label{prop:first_order_optimality}
	Let $\bar{u} \in \mathcal{U}_{\mathrm{ ad}}$ be an optimal control with associated state $\bar{y} = y_{\bar{u}}$ and adjoint state $\bar{\phi} \in \mathcal{Y}$ satisfying
	\begin{equation}
		\label{eq:adjoint_optimal}
		\begin{cases}
			\ddot{\bar{\phi}} - \dot{\bar{\phi}} = \Delta \bar{\phi} + \bar{u}\, \bar{\phi} + \bar{y} - y_d, & \text{in } Q, \\
			\bar{\phi} = 0, & \text{on } \Sigma, \\
			\|\bar{\phi}(t)\| + \|\dot{\bar{\phi}}(t)\|_{H^{-1}(\Omega)} \to 0, & \text{as } t \to \infty.
		\end{cases}
	\end{equation}
	Then $\bar{u}$ satisfies the variational inequality
	\begin{equation}
		\label{eq:variational_inequality}
		\int_0^\infty \int_\Omega \bigl( \bar{\phi} \bar{y} + \gamma\bar{u} \bigr) (u - \bar{u}) \, \mathrm{d}x \, \mathrm{d}t \ge 0
		\quad \text{for all } u \in \mathcal{U}_{\mathrm{ ad}}.
	\end{equation}
	In particular, the optimal control $\bar{u}$ is characterized pointwise almost everywhere by
	\begin{equation}
		\label{eq:projection_formula}
		\bar{u}(t,x) = \Pi_{[\alpha(t,x), \beta(t,x)]}\bigl(-\tfrac{1}{\gamma}\bar{\phi}(t,x) \bar{y}(t,x)\bigr),
	\end{equation}
	where $\Pi_{[\alpha,\beta]}$ denotes the projection onto the admissible interval $[\alpha(t,x),\beta(t,x)]$.
\end{proposition}

\begin{proof}
	Since $\bar{u}$ minimizes $J$ over $\mathcal{U}_{ad}$, the variational inequality 
	\[
	J'(\bar{u})(u - \bar{u}) \ge 0 \qquad \forall\, u \in \mathcal{U}_{ad}
	\]
	holds. By Proposition~\ref{prop:cost_differentiability}, the derivative of $J$ at $\bar{u}$ is
	\begin{equation}
		\label{eq:cost_derivative_adjoint}
		J'(\bar{u})(h)
		= \int_0^\infty \int_\Omega \bigl(\bar{\phi}\,\bar{y} + \gamma \bar{u}\bigr)\, h 
		\,\mathrm{d}x\,\mathrm{d}t,
	\end{equation}
	and choosing $h = u - \bar{u}$ yields \eqref{eq:variational_inequality}.
	
	To obtain the pointwise characterization, observe that \eqref{eq:variational_inequality} is equivalent to the scalar condition
	\[
	(\bar{\phi}(t,x)\,\bar{y}(t,x) + \gamma \bar{u}(t,x))
	\,(v - \bar{u}(t,x)) \ge 0 
	\qquad 
	\forall\, v \in [\alpha(t,x),\beta(t,x)]
	\]
	for almost every $(t,x)\in Q$.  
	This is precisely the well-known optimality condition for a pointwise projection onto a closed interval; see, for example, \cite{troeltzsch2010optimal}.  
	Hence,
	\[
	\bar{u}(t,x)
	= 
	\Pi_{[\alpha(t,x),\beta(t,x)]}
	\!\left(-\frac{1}{\gamma}\,\bar{\phi}(t,x)\,\bar{y}(t,x)\right),
	\]
	which is \eqref{eq:projection_formula}.
\end{proof}

\section{Second-order optimality conditions} \label{sec:secondorder}
Define the active sets at $\bar{u}$ by
\[
A_\alpha := \{(x,t) \in Q \;:\; \bar{u}(x,t) = \alpha(x,t)\}, 
\qquad 
A_\beta := \{(x,t) \in Q \;:\; \bar{u}(x,t) = \beta(x,t)\}.
\]

The cone of critical directions at $\bar{u}$ is defined by
\[
\mathcal{C}(\bar{u}) :=
\left\{
h \in \mathcal{U} \;\middle|\;
J'(\bar{u})[h] = 0,\;
\begin{cases}
	h \ge 0 \quad \text{a.e. on } A_\alpha,\\
	h \le 0 \quad \text{a.e. on } A_\beta
\end{cases}
\right\}.
\]

\begin{lemma}
	The critical cone $\mathcal{C}(\bar{u})$ is closed in $\mathcal{U}$.
\end{lemma}

\begin{proof}
	The cone $\mathcal{C}(\bar{u})$ is the intersection of three sets:
	\[
	\mathcal{C}(\bar{u})
	= \ker J'(\bar{u})
	\;\cap\;
	\{h \in \mathcal{U} : h \ge 0 \text{ a.e. on } A_\alpha\}
	\;\cap\;
	\{h \in \mathcal{U} : h \le 0 \text{ a.e. on } A_\beta\}.
	\]	
	The first set is closed since
	$J'(\bar{u}) : \mathcal{U} \to \mathbb{R}$
	is continuous by estimate~\eqref{lem:estimate1}.
	The remaining two sets are closed in $L^\infty(Q)$, hence in $\mathcal{U}$,
	as uniform convergence preserves pointwise a.e.\ inequalities.
\end{proof}

\begin{theorem}[Second-order necessary conditions]
	\label{thm:second_order_necessary}
	If $\bar{u} \in \mathcal{U}_{\mathrm{ ad}}$ is a locally optimal control, then
	\[
	J''(\bar{u})[h,h] \ge 0
	\quad\forall\, h \in \mathcal{C}(\bar u).
	\]
\end{theorem}

\begin{proof}
	Let $h\in\mathcal{C}(\bar u)$.  
	The sign conditions in the definition of $\mathcal{C}(\bar u)$ ensure that  
	$\bar u+\varepsilon h\in\mathcal{U}_{ad}$ for all sufficiently small $\varepsilon>0$.\\	
	Since $\bar u$ is a local minimizer,
	\[
	J(\bar u)\le J(\bar u+\varepsilon h) \qquad \text{for small }\varepsilon>0.
	\]	
	The second-order Taylor expansion of $J$ gives
	\[
	J(\bar u+\varepsilon h)
	= J(\bar u)
	+ \varepsilon\, J'(\bar u)[h]
	+ \frac{\varepsilon^{2}}{2}J''(\bar u)[h,h]
	+ o(\varepsilon^{2})
	\quad (\varepsilon\to 0^+).
	\]	
	Since $h\in\mathcal{C}(\bar u)$, we have the stationarity condition
	\[
	J'(\bar u)[h]=0.
	\]	
	Thus
	\[
	0 \le J(\bar u+\varepsilon h)-J(\bar u)
	= \frac{\varepsilon^{2}}{2}J''(\bar u)[h,h] + o(\varepsilon^{2}).
	\]	
	Divide by $\varepsilon^{2}/2$ and let $\varepsilon\to 0^+$:
	\[
	J''(\bar u)[h,h]\ge 0.
	\]
\end{proof}

\begin{theorem}[Second-order sufficient conditions]
	Let \( \bar{u} \in \mathcal{U}_{\mathrm{ad}} \) be a control satisfying the first-order condition \eqref{eq:variational_inequality}, and assume there exists \( \mu > 0 \) such that
	\begin{equation}\label{suf4}
		{J}''(\bar{u})[h,h] \ge \mu \|h\|^2_{L^2(Q)}, \quad \forall h \in  \mathcal{C}(\bar u).
	\end{equation}
	Then, there exist \( \delta > 0 \) and \( \rho > 0 \) such that the quadratic growth condition holds:
	\begin{equation}\label{sufff}
		{J}(\bar{u}) + \delta \|u - \bar{u}\|_{L^2(Q)}^2 \leq {J}(u), \quad \forall u \in \mathcal{U}_{\mathrm{ad}} \ \text{ with }\ \|\bar{u} - u\|_{\mathcal U} \le \rho.
	\end{equation}
	In particular, \( \bar{u} \) is a strict local solution to  \eqref{prob:P}.
\end{theorem}

\begin{proof}
	We proceed by contradiction. Assume the conclusion fails: for every integer $n \geq 1$, there exists $u_n \in \mathcal{U}_{\mathrm{ad}}$ such that
	\begin{equation}\label{contrary_assumption}
		0 \neq \|u_n - \bar{u}\|_{\mathcal{U}} < \frac{1}{n} \quad \text{and} \quad 
		J(\bar{u}) + \frac{1}{2n}\|u_n - \bar{u}\|_{L^2(Q)}^2 > J(u_n).
	\end{equation}	
	Define the normalized direction $h_n = \frac{u_n - \bar{u}}{\|u_n - \bar{u}\|_{L^2(Q)}}$ and set $\lambda_n = \|u_n - \bar{u}\|_{L^2(Q)}$. By construction, $\|h_n\|_{L^2(Q)} = 1$ for all $n$.\\	
	Since $(h_n)$ is bounded in $\mathcal{U}$, there exists a subsequence (still denoted $h_n$) converging weakly-$*$ to some $h \in \mathcal{U}$.
	
	\medskip
	\noindent\textbf{Step 1:} We show that $J'(\bar{u})[h] = 0$.\\	
	From the first-order condition \eqref{eq:variational_inequality} and estimate \eqref{lem:estimate1}:
	\[
	J'(\bar{u})[h] = \lim_{n \to \infty} J'(\bar{u})[h_n] 
	= \lim_{n \to \infty} \frac{1}{\lambda_n} J'(\bar{u})[u_n - \bar{u}] \geq 0.
	\]	
	For the reverse inequality, by the mean value theorem:
	\begin{equation}\label{mvt_expansion}
		\frac{J(\bar{u} + \lambda_n h_n) - J(\bar{u})}{\lambda_n} = J'(\bar{u} + \theta_n(u_n - \bar{u}))[h_n]
	\end{equation}
	for some $\theta_n \in (0,1)$. This equals
	\[
	J'(\bar{u})[h_n] + \bigl[J'(\bar{u} + \theta_n(u_n - \bar{u})) - J'(\bar{u})\bigr][h_n].
	\]	
	By estimate \eqref{lem:estimate3}:
	\[
	\bigl|[J'(\bar{u} + \theta_n(u_n - \bar{u})) - J'(\bar{u})][h_n]\bigr| 
	\leq C\theta_n\|u_n - \bar{u}\|_{\mathcal{U}}\|h_n\|_{\mathcal{U}} 
	\leq \frac{C}{n}\|h_n\|_{\mathcal{U}} \to 0.
	\]	
	Taking limits in \eqref{mvt_expansion} and using \eqref{contrary_assumption}:
	\[
	J'(\bar{u})[h] = \lim_{n \to \infty} \frac{J(u_n) - J(\bar{u})}{\lambda_n} 
	\leq \lim_{n \to \infty} \frac{-\frac{1}{2n}\lambda_n^2}{\lambda_n} 
	= \lim_{n \to \infty} \frac{-\lambda_n}{2n} = 0.
	\]	
	Thus $J'(\bar{u})[h] = 0$.
	
	\medskip
	\noindent\textbf{Step 2:} We show that $h \in \mathcal{C}(\bar{u})$.\\	
	Each $h_n = \frac{u_n - \bar{u}}{\lambda_n}$ satisfies:
	\begin{itemize}
		\item On $A_\alpha$: since $u_n \geq \alpha = \bar{u}$, we have $h_n \geq 0$ a.e.
		\item On $A_\beta$: since $u_n \leq \beta = \bar{u}$, we have $h_n \leq 0$ a.e.
	\end{itemize}	
	Since $\mathcal{C}(\bar{u})$ is closed and $h_n \rightharpoonup^* h$, the weak-$*$ limit $h$ inherits these sign conditions. Combined with $J'(\bar{u})[h] = 0$, we conclude $h \in \mathcal{C}(\bar{u})$.
	
	\medskip
	\noindent\textbf{Step 3:} We derive the contradiction.\\	
	By Taylor expansion:
	\[
	J(u_n) - J(\bar{u}) = \lambda_n J'(\bar{u})[h_n] + \frac{\lambda_n^2}{2}\bigl[J''(\bar{u})[h_n, h_n] + r_n\bigr],
	\]
	where
	\[
	r_n = \bigl[J''(\bar{u} + \theta_n(u_n - \bar{u})) - J''(\bar{u})\bigr][h_n, h_n]
	\]
	for some $\theta_n \in (0,1)$.\\	
	From \eqref{contrary_assumption} and $J'(\bar{u})[u_n - \bar{u}] \geq 0$:
	\[
	\frac{\lambda_n^2}{2n} > J(u_n) - J(\bar{u}) \geq \frac{\lambda_n^2}{2}\bigl[J''(\bar{u})[h_n, h_n] + r_n\bigr].
	\]	
	Dividing by $\frac{\lambda_n^2}{2}$:
	\begin{equation}\label{second_deriv_bound}
		J''(\bar{u})[h_n, h_n] + r_n < \frac{1}{n}.
	\end{equation}	
	By estimate \eqref{lem:estimate4}:
	\[
	|r_n| \leq C\|u_n - \bar{u}\|_{\mathcal{U}}\|h_n\|_{\mathcal{U}}^2 
	\leq \frac{C}{n}\|h_n\|_{\mathcal{U}}^2 \to 0.
	\]	
	Thus from \eqref{second_deriv_bound}:
	\[
	\limsup_{n \to \infty} J''(\bar{u})[h_n, h_n] \leq 0.
	\]	
	By weak lower semicontinuity of $J''(\bar{u})[\cdot, \cdot]$:
	\[
	0 \leq J''(\bar{u})[h, h] \leq \liminf_{n \to \infty} J''(\bar{u})[h_n, h_n] \leq 0,
	\]
	so $J''(\bar{u})[h, h] = 0$ and $J''(\bar{u})[h_n, h_n] \to 0$.\\	
	Applying the coercivity assumption \eqref{suf4}:
	\[
	\mu = \mu\|h_n\|_{L^2(Q)}^2  \leq J''(\bar{u})[h_n, h_n] \to 0,
	\]
	which contradicts $\mu > 0$.\\	
	Therefore, the quadratic growth condition \eqref{sufff} holds for some $\delta > 0$ and $\rho > 0$, establishing that $\bar{u}$ is a strict local solution to \eqref{prob:P}.
\end{proof}

\section{Conclusion}
\label{sec:conclusion}

This paper has addressed the bilinear optimal control problem for the damped wave equation over an infinite time horizon, providing a comprehensive theoretical framework for this class of problems. Well-posedness of the state and adjoint equations, existence of optimal controls, and both necessary and sufficient optimality conditions have been established. The analysis extends finite-horizon results \cite{bethke2018sufficient,zerrik2019regional} to the more challenging setting of an unbounded time domain, where specialized techniques are required to handle asymptotic behavior and ensure the finiteness of the cost functional, and goes beyond previous work on additive control problems \cite{kunisch2016optimal,kroner2011semismooth} by addressing the bilinear control structure.
Future work may explore more general boundary conditions,  stochastic situations,   and efficient numerical schemes for practical applications such as quantum control, population dynamics, and vibration management in elastic structures.

\end{document}